\documentclass[12pt]{amsart}
\usepackage{latexsym,amsxtra,amscd,ifthen,amsmath,color, multicol,hyperref,mathdots,mathrsfs}
\usepackage{amsfonts}
\usepackage{verbatim}
\usepackage{amsmath}
\usepackage{amsthm}
\usepackage{amssymb}
\usepackage[all,cmtip]{xy}
\usepackage[normalem]{ulem}
\usepackage{enumerate}
\usepackage[latin1]{inputenc}
\usepackage{graphicx}
\usepackage{multirow}
\usepackage{latexsym}
\swapnumbers
\theoremstyle{plain}
\newtheorem{thm}{Theorem}[section]
\newtheorem{lem}[thm]{Lemma}
\newtheorem{prop}[thm]{Proposition}

\theoremstyle{definition}
\newtheorem{dfn}[thm]{Definition}

\newtheorem*{Ack}{Acknowledgement}

\newtheorem{Radford}[thm]{Radford algebras}

\theoremstyle{remark}

\def\k{\ensuremath{\bold{k}}}

\def\B{\mathfrak{B}}

\newcommand*{\e}{\ensuremath{\varepsilon}}

\newcommand*{\field}{\ensuremath{\bold{k}}}

\def\dim{\operatorname{dim}}

\def\Chara{\operatorname{char}}

\begin{document}

\title{Computing indicators of Radford algebras}

\author{Hao Hu}
\email{hah73@pitt.edu}
\author{Xinyi Hu}
\email{xih40@pitt.edu}
\author{Linhong Wang}
\email{lhwang@pitt.edu}
\address{Department of Mathematics\\
University of Pittsburgh, Pittsburgh, PA 15260}

\author{Xingting Wang}
\email{xingting@temple.edu}
\address{Department of Mathematics\\
Temple University,
Philadelphia, PA 19122}

\begin{abstract}
We compute higher Frobenius-Schur indicators of Radford algebras in positive characteristic and find minimal polynomials of these linearly recursive sequences. As a result of Kashina, Montgomery and Ng, we obtain gauge invariants for the monoidal categories of representations of Radford algebras.
\end{abstract}


\maketitle

\section{Introduction}

In group theory, the Frobenius-Schur (FS) indicator provides a criterion, depending on its possible values $1$, $0$, or $-1$, for determining whether an irreducible representation of a finite group $G$ is real, complex or quaternionic. This result was generalized to any semisimple Hopf algebra over an algebraically closed field of characteristic zero in \cite{LinMon}. Kashina-Montgomery-Ng in \cite{KMN} proposed a definition of higher Frobenius-Schur (FS) indicators for an arbitrary finite-dimensional Hopf algebra, which further generalizes the notion given in \cite{KSZ2} regarding the regular representation of a semisimple Hopf algebra. Moreover, they proved that these indicators are gauge invariant under gauge equivalence in the sense of \cite{Kas}.
Later, the properties of these indicators were further discussed by Shimizu \cite{Shi} mainly focusing on the complex Hopf algebras.

The definition of higher FS indicators of the regular representation of a finite-dimensional Hopf algebra is straightforward by taking the trace of the Sweedler powers followed by the antipode; see \cite[Definition 2.1]{KMN}. But to find their values can be arithmetically challenging over the complex numbers, e.g, the indicators of Taft algebras; see\cite[\S 3]{KMN}. Besides Taft algebra, another well-studied Hopf algebra with simple defining relation is the Radford algebra $R(p)$, which was introduced in \cite[4.13]{Rad} and is over a base algebraically closed field of prime characteristic $p$. It was proved in \cite{WW} that $R(p)$  is  the only noncommutative and noncocommutative pointed Hopf algebra of dimension $p^2$.

In this short note, we find that the higher FS indicators of the Radford algebra $R(p)$ are
\[\Big\{\nu_n(R(p))\Big\}_{ n\geq 1}=\Big\{\underbrace{1, \ldots, 1}_{p-1}, 0, \underbrace{1, \ldots, 1}_{p-1}, 0, \ldots\Big\}.\]
Our approach is via concrete computation involving the left integrals of the Radford algebra and those of its dual Hopf algebra. Our result verified,  in the case of Radford algebra, a theorem by Shimizu on higher FS indicators over positive characteristic, which states that  the sequence of indicators always appears periodically in positive characteristic \cite[Corollary 4.6]{Shi}. As a result of Kashina, Montgomery and Ng, we obtain gauge invariants for the monoidal category of the representation of Radford algebras. Moreover, we also find the minimal polynomial of the sequence of indicators of the Radford algebra.

\begin{Ack}
We began this work in an undergraduate research project at the University of Pittsburgh. We would like to express our gratitude to the math department for hosting a visit of the fourth author in Spring 2016. We are grateful to the referee for her/his careful reading.
\end{Ack}

\section{Preliminary}\label{S:1}
Throughout, $\k$ is an algebraically closed field, $H$ is a finite-dimensional Hopf algebra over $\k$.
We use the standard notation $(H,\,m,\,u,\,\Delta,\,\e,\,S)$, where $m: H\otimes H \to H$ is the multiplication map, $u: \k\to H$ is the unit map, $\Delta: H\to H\otimes H$ is the comultiplication map, $\e: H\to \k$ is the counit map, and $S: H\to H$ is the antipode. The vector space dual of $H$ is also a Hopf algebra and will be denoted by $H^*$. The bialgebra maps and antipode of $H^*$ are given by $(m_{H^*},\,u_{H^*},\,\Delta_{H^*},\,\e_{H^*},\,S_{H^*})=(\Delta^*,\,\e^*,\,m^*,\,u^*,\,S^*)$, where $^*$ is the transpose.
We use the  Sweedler notation $\Delta(h)=\sum h_{(1)}\otimes h_{(2)}$. If $f,\,g\in H^*$, then $fg(h)=\sum f\left(h_{(1)}\right)g\left(h_{(2)}\right)$ for any $h\in H$ and $\e_{H^*}(f)=f(1)$.

\begin{dfn}\cite[Definition 2.1.1]{Mon}
A left integral in $H$ is an element $\Lambda\in H$ such that $h\Lambda=\e(h)\Lambda$, for all $h\in H$; a right integral in $H$ is an element $\Lambda'\in H$ such that $\Lambda'h=\e(h)\Lambda'$ for all $h\in H$. The space of left integrals and the space of right integrals are denoted by $\int^l_H$ and $\int^r_H$, respectively.
\end{dfn}

\begin{lem}\cite[Theorem 2.1.3]{Mon} $\int^l_H$ and $\int^r_H$ are each one-dimensional.
\end{lem}

\begin{lem}\label{L:intDual}
Suppose $\lambda\in H^*$. Then $\lambda$ is a left integral of $H^*$ if and only if $\sum h_{(1)}\lambda (h_{(2)})=\lambda (h)$ for any $h\in H$. A similar criterion holds for a right integral of $H^*$,i.e., $\lambda$ is a right integral of $H^*$ if and only if $\sum \lambda (h_{(1)})h_{(2)}=\lambda (h)$ for any $h\in H$.
\end{lem}

\begin{proof}
By definition, $\lambda$ is a left integral in $H^*$ if and only if $f\lambda = \e_{H^*}(f) \lambda$ for any linear function $f\in H^*$. That is, $f\lambda(h)=\e_{H^*}(f)\lambda(h)$ for any $h\in H$. By duality, this is equivalent to $\sum f\big(h_{(1)}\big)\lambda \big(h_{(2)}\big)=f(1)\lambda(h)$ or  $ f\big(\sum h_{(1)}\lambda(h_{(2)})\big)$ $=f\big(1\lambda(h)\big)$ since $f$ is linear. Note that $f$ is arbitrary in $H^*$. We have $\lambda$ is a left integral in $H^*$ if and only if $\sum h_{(1)}\lambda (h_{(2)})=\lambda (h)$ for any $h\in H$. The proof for right integrals is the same.
\end{proof}

\begin{dfn}\cite[Definition 2.1]{KMN}
Let $n$ be a positive integer. Suppose $h_1, \ldots$, $h_n \in H$. Then the $n$-th power of multiplication is defined as
\[
m^{(n)}(h_1\otimes \cdots \otimes h_n)=h_1\cdots h_n.
\]
Let $h\in H$. The $n$-th power of comultiplication is defined to be
\[
\Delta^{(n)}(h)=\begin{cases}
h & n=1\\
(\Delta^{(n-1)}\otimes \mathrm{id})\left(\Delta(h)\right) & n\geq 2
\end{cases}
\]
The $n$-th Sweedler power of $h$ is defined to be \[P_n(h)=h^{[n]}=
\begin{cases}
\epsilon(h) 1_H & n=0\\
m^{(n)}\circ\Delta^{(n)}(h) & n\geq 1
\end{cases}
\]
The $n$-th indicator of $H$ is given by  $$\nu_n(H)=\mathrm{Tr}\big(S\circ P_{n-1}\big).$$ In particular, $\nu_1(H)=1$ and $\nu_2(H)=\mathrm{Tr}(S)$.
\end{dfn}


Let $H$ and $K$ be two finite-dimensional Hopf algebras over $\field$ such that the two representation categories $\text{Rep}(H)$ and $\text{Rep}(K)$ are monoidally equivalent. By \cite[Theorem 2.2]{NgSch},  $H\cong K^F$, where $K^F$ is a Drinfeld twist by a gauge transformation $F$ on $H$ which satisfies some $2$-cocycle conditions. Then $H$ and $K$ are said to be \emph{gauge equivalent} Hopf algebras.

\begin{thm}\cite[Theorem 2.2 \& Corollary 2.6]{KMN}\label{T:FS}
The sequence $\{\nu_n(H)\}$ is an invariant of the gauge equivalence class of Hopf algebras of $H$, that is, if $H$ and $K$ are gauge equivalent then $\{\nu_n(H)\}=\{\nu_n(K)\}$. Suppose $\lambda\in H^*$ and $\Lambda\in H$ are both left integrals (or both right integrals) such that $\lambda(\Lambda)=1$. Then
\[
\nu_n(H)=\lambda \left(\Lambda^{[n]}\right)
\]
for all positive integers $n$.
\end{thm}

\begin{prop}\cite[Corollary 4.6]{Shi}
Suppose $\Chara \k >0$. Then, for any finite dimensional Hopf algebra $H$ over $\k$, the sequence $\{\nu_n(H)\}$ is periodic.
\end{prop}

\begin{dfn}
A sequence $\{a_n\}_{n\ge 1}$ is linearly recursive if there exists a non-zero polynomial $f(x)=f_0+f_1x+f_{m-1}x^{m-1}+f_mx^m$ such that
\[
f_0a_n+f_1a_{n+1}+\cdots+f_ma_{m+n}=0,
\]
for any positive integer $n$. In such a case, we say that $\{a_n\}_{n\ge 1}$ satisfies the polynomial $f(x)$. The monic polynomial of the least degree satisfied by a linearly recursive sequence is called the minimal polynomial of the sequence.
\end{dfn}

\begin{prop}\cite[Proposition 2.7]{KMN}
The sequence $\{\nu_n(H)\}$ is linearly recursive and the degree of its minimal polynomial is at most $(\dim H)^2$. The minimal polynomial is also a gauge invariant, that is, if $H$ and $K$ are gauge equivalent, then $\{\nu_n(H)\}$ and $\{\nu_n(K)\}$ have the same monic minimal polynomial.
\end{prop}

Next, we consider a free bialgebra $\B$ and the comultiplication of certain monomials in $\B$. This information will be used later in our computation of indicators of $R(p)$.

\begin{dfn}
Let $\B=\field\langle g,x\rangle$ be the free $\field$-algebra on two generators $g$ and $x$. Equipped with the comultiplication and the counit given by
\[
\Delta(g)=g\otimes g\ \quad \quad \ \Delta(x)=x\otimes 1+g\otimes x
\quad \quad
\e(g)=1\ \quad \text{and}\ \e(x)=0,
\]
the free algebra becomes the free bialgebra $(\B, \, \Delta,\, \e)$. Let $C_{k,l}$ denote the sum of all monomials with $k$ $g$'s and $l$ $x$'s, and $C_{0,0}=1$ and $C_{k,l}=0$ if $k$ or $l<0$ by convention.
\end{dfn}

\begin{lem}\label{L:B}
In the free bialgebra $\B$, we have
\begin{itemize}
\item[(1)] $C_{k,l}=g\,C_{k-1,l}+x\,C_{k,l-1}=C_{k-1,l}\,g+C_{k,l-1}\,x$.
\item[(2)] $\Delta(x^n)=\sum_{k\ge 0} C_{k,n-k}\otimes x^k$ for $n\ge 0$.
\item[(3)] $\Delta(C_{p,q})=\sum_{k\ge 0} C_{p+k,q-k}\otimes C_{p,k}$.
\end{itemize}
\end{lem}

\begin{proof}
(1) is clear, since the leftmost (rightmost) factor of any monomial in the sum $C_{k,l}$ is either $g$ or $x$. For (2), we prove by induction. When $n=0$, $\sum_{k\ge 0} C_{k,n-k}\otimes x^k=C_{0,0}\otimes 1=1\otimes 1=\Delta(1)$. When $n=1$, $\sum_{k\ge 0} C_{k,n-k}\otimes x^k=C_{0,1}\otimes 1+C_{1,0}\otimes x=x\otimes 1+g\otimes x=\Delta(x)$. Suppose $\Delta(x^n)=\sum_{k\ge 0} C_{k,n-k}\otimes x^k$. Then
\begin{align*}
\Delta(x^{n+1})&=\Delta(x^n)\Delta(x)=\big(\sum_{k\ge 0} C_{k,n-k}\otimes x^k\big)\cdot \big(x\otimes 1+g\otimes x\big)\\
&=\sum_{k\ge 0} C_{k,n-k}x\otimes x^k+ \sum_{k\ge 1} C_{k-1,n-k+1}g\otimes x^k\\
&=\sum_{k\ge 0} C_{k,n-k}x\otimes x^k+ \sum_{k\ge 1}\big( C_{k,n-k+1}-C_{k, n-k} x\big)\otimes x^k\\
&=\sum_{k\ge 0} C_{k,(n+1)-k}\otimes x^k.
\end{align*}
To show (3), we use the fact that
\[(\Delta\otimes \mathrm{id})(\Delta(x^n))=(\mathrm{id}\otimes \Delta)(\Delta(x^n)).\]
By (2), we have
\[(\Delta\otimes \mathrm{id})(\Delta(x^n))=\sum_{k\geq 0}\Delta(C_{k,n-k})\otimes x^k=\sum_{p+q=n}\Delta(C_{p,q})\otimes x^p.\]
On the other hand,
{\small
\begin{align*}
(\mathrm{id}\otimes \Delta)(\Delta(x^n))&=\sum_{l\geq 0} C_{l,n-l}\otimes \Delta(x^l)=\sum_{l\geq 0}C_{l, n-l}\otimes \Big(\sum_{p\geq 0}C_{p, l-p}\otimes x^p\Big)\\
&=\sum_{p\geq 0}\sum_{l\geq p}C_{l, n-l}\otimes C_{p, l-p}\otimes x^p\\
&=\sum_{p+q=n}\Big(\sum_{l-p=k\geq 0}C_{p+k, q-k}\otimes C_{p,k}\Big)\otimes x^p.
\end{align*}
}
It then follows that $\Delta(C_{p,q})=\sum_{k\ge 0} C_{p+k,q-k}\otimes C_{p,k}$.
\end{proof}

\begin{lem}\label{L:SP}
In the free bialgebra $\B$, we have
\begin{align*}
\left(g^ix^j\right)^{[n+1]}=\sum_{0\le k_1+\cdots+k_n\le j}    &g^iC_{k_1+\cdots+k_n,j-(k_1+\cdots+k_n)}g^iC_{k_1+\cdots+k_{n-1},k_n}\\
&\cdots g^iC_{k_1,k_2}g^iC_{0,k_1}.
\end{align*}
\end{lem}

\begin{proof}
By induction on $n$, using Lemma \ref{L:B}, it is easy to see that
\begin{align*}
\Delta^{(n+1)}\left(C_{p,q}\right)=\notag\sum_{0\le k_1+\cdots+k_n\le q} &C_{p+k_1+\cdots+k_n,q-(k_1+\cdots+k_n)}\otimes  \\
&C_{p+k_1+\cdots+k_{n-1},k_n}\otimes \cdots \otimes C_{p+k_1,k_2}\otimes C_{p,k_1}.
\end{align*}
Therefore, we have
\begin{align*}
\left(g^ix^j\right)^{[n+1]}&=m^{(n+1)}\left(\Delta^{(n+1)}(g^i)\Delta^{(n+1)}(x^j)\right)\\
&=m^{(n+1)}\left(g^i\otimes \cdots \otimes g^i\right)\left(\sum_{k\ge 0} \Delta^{(n)}(C_{k,j-k})\otimes x^k\right)\\
&=\sum_{0\le k_1+\cdots+k_n\le j} g^iC_{k_1+\cdots+k_n,j-(k_1+\cdots+k_n)}\cdots g^iC_{k_1,k_2}g^iC_{0,k_1}.
\end{align*}
\end{proof}

\section{Radford algebras}
In this section, the base field $\field$ is algebraically closed of prime characteristic $p$.

\begin{Radford}
The Radford algebra $R(p)$ was discussed in \cite[4.13]{Rad} over a base field $\field$ of prime characteristic $p$, and was proved in \cite{WW} to be the only noncommutative and noncocommutative pointed Hopf algebra of dimension $p^2$ over $\field$. In fact, one can write $R(p)$ as the quotient Hopf algebra $\B/\mathscr R$, where the ideal $\mathscr R$ of $\B$ is generated by, if $p>2$,
\begin{align}\label{R:R}
g^p-1,\  x^p-x,\ [g,x]-\left(g^2-g\right), \tag{$\mathscr R$}
\end{align}
or
\begin{align*}
g^2-1,\  x^2-x,\ [g,x]-\left(1-g\right),
\end{align*}
if $p=2$.
It is straightforward to check that the Radford algebra $R(p)$ has dimension $p^2$ and the linear basis can be chosen as $\{g^ix^j\, |\, 0\le i,j\le p-1\}$.
We denote by $c_{k,l}$ the image of $C_{k,l}$ (the sum of all monomials with $k$ $g$'s and $l$ $x$'s in $\B$) in $R(p)$ under the projection $\B\to \B/\mathscr R=R(p)$. It follows from \eqref{R:R} that, for $0\le k, l\le p-1$,
\begin{align}\label{E:R1}
c_{k,l}={k+l\choose k}g^kx^l+\sum_{\substack{0\le i\le p-1\\ 0\le j \le l-1}} a_{ij} g^{i}x^{j},\ \text{for some}\ a_{ij}\in \field.
\end{align}
Moreover, the Radford algebra $R(p)$ is self-dual.
The dual basis of $\left(R(p)\right)^*$ to the chosen basis $\{g^ix^j\, |\, 0\le i,j\le p-1\}$ of $R(p)$ is $\{\delta_{g^ix^j}\, |\, 0\leq i, j\leq p-1\}$, where $\delta_{g^ix^j}$ are characteristic functions, that is,
\[
\delta_{g^ix^j}(g^m x^n)=
\begin{cases}
1 & \text{if}\ m=i,\; n=j\\
0 & \text{otherwise}.
\end{cases}
\]
\end{Radford}

\begin{lem}\label{L:R}
For the Radford algebra $R(p)$, the integral spaces are given by
\begin{align*}
\int^l_{R(p)} &=\field \left(\sum_{0\le i\le p-1}g^i\right)\left(\sum_{1\le i\le p-1} (-1)^ix^{i}\right),\\
\int^r_{R(p)} &=\field \left(\sum_{1\le i\le p-1} x^{i}\right)\ \left(\sum_{0\le i\le p-1} g^i\right).
\end{align*}
For the dual Hopf algebra $\left(R(p)\right)^*$, the integral spaces are given by
\begin{align*}
\int^l_{R(p)^*}=\field\, \delta_{gx^{p-1}}\ \text{and}\ \int^r_{R(p)^*}=\field\, \delta_{x^{p-1}},
\end{align*}
\end{lem}

\begin{proof}
Note that $\e(g)=1$, $\e(x)=0$, and $\e$ is linear. To show that the element $\Lambda=\left(\sum_{0\le i\le p-1}g^i\right)\left(\sum_{1\le i\le p-1} (-1)^ix^{i}\right)$ is a left integral in $R(p)$, it is sufficient to show that $g\Lambda=\Lambda$ and $x\Lambda=0$. The first equation is obvious. One can check that $[x, g^i]=ig^i(1-g)$. Hence we have
\[
\left[x,\quad \sum_{i=1}^{p-1}g^i\right]
=\sum_{i=1}^{p-1}ig^i(1-g)
=\sum_{i=1}^{p-1}ig^i-\sum_{j=2}^{p}(j-1)g^{j}
=g+\sum_{i=2}^{p-1}g^i+g^p
=\sum_{i=0}^{p-1}g^i,
\]
and so
\begin{align*}
x\Lambda&=x\left(\sum_{0\le i\le p-1}g^i\right)\left(\sum_{1\le i\le p-1} (-1)^ix^{i}\right)\\
&=\left(\sum_{i=0}^{p-1}g^i\right)(x+1)\left(\sum_{1\le i\le p-1} (-1)^ix^{i}\right)=\left(\sum_{i=0}^{p-1}g^i\right)(x^p-x)=0.
\end{align*}
Therefore, $\Lambda$ is a left integral in $R(p)$.

To show that the characteristic function $\delta_{gx^{p-1}}$ is a left integral in $R(p)^*$, it is sufficient, by Lemma \ref{L:intDual}, to verify that
\[\sum h_{(1)}\delta_{gx^{p-1}} (h_{(2)})=\delta_{gx^{p-1}} (h),
\text{ for }h=g^ix^j\in R(p) \text{ with } 0\leq i, j \leq p-1.\]
By Lemma \ref{L:B}, we have
$\Delta(g^ix^j)=\left(g^i\otimes g^i\right)\Delta(x^j)=\sum_{k=0}^j g^i  c_{k,\, j-k}\otimes g^ix^k$.
Hence
\begin{align*}
\sum h_{(1)}\delta_{gx^{p-1}} (h_{(2)})&=\sum_{k=0}^j \Big(g^i  c_{k,\, j-k}\cdot \delta_{gx^{p-1}} (g^ix^k)\Big)\\
&=
\begin{cases}
gc_{p-1, 0}=1 & \text{if}\ i=1,\; j=k=p-1\\
0 & \text{otherwise}.
\end{cases}
\end{align*}
On the other hand,
\[
\delta_{gx^{p-1}} (h)=\delta_{gx^{p-1}}(g^i x^j)=
\begin{cases}
1 & \text{if}\ i=1,\; j=p-1\\
0 & \text{otherwise}.
\end{cases}
\]
Therefore, $\delta_{gx^{p-1}}$ is a left integral in $R(p)^*$. The statements on right integrals can be shown similarly.
\end{proof}

\begin{thm}\label{T:R}
The higher FS indicators of the Radford algebra $R(p)$ are given by
\begin{align*}
\nu_n(R(p))=
\begin{cases}
1 & \text{if}\ n\not\equiv 0 \pmod p\\
0 & \text{if}\ n\equiv 0 \pmod p.
\end{cases}
\end{align*}
\end{thm}
\begin{proof}
By Lemma \ref{L:R}, we choose the left integral $\lambda=\delta_{gx^{p-1}}$ of the dual Hopf algebra $\left(R(p)\right)^*$, and the left integral
$
\Lambda=\left(\sum_{0\le i\le p-1}g^i\right)\left(\sum_{1\le i\le p-1} (-1)^ix^i\right)
 $
of $R(p)$. It is clear that $\lambda(\Lambda)=1$. By Theorem \ref{T:FS}, we have
\begin{align*}
\nu_{n+1}\left(R(p)\right)=\lambda\left(\Lambda^{[n+1]}\right)=\delta_{gx^{p-1}}\left(\sum_{0\le i,j\le p-1} (-1)^j(g^ix^{j})^{[n+1]}\right).
\end{align*}
By Lemma \ref{L:SP} and \eqref{E:R1}, one sees that, for any $0\le i,j\le p-1$,
\begin{align*}
\left(g^ix^j\right)^{[n+1]}\in \text{Span}\left(g^kx^l\, \big|\, 0\le k\le p-1,0\le l\le j\right).
\end{align*}
Hence,
\[\nu_{n+1}\left(R(p)\right)=\delta_{gx^{p-1}}\left(\sum_{0\le i\le p-1} (g^ix^{p-1})^{[n+1]}\right).\]
Suppose $k_1, \ldots,k_n$ are nonnegative integers such that $\sum_{i=1}^nk_i=m$. Recall that the multinomial coefficients are given by
\[
{m \choose k_1, \ldots, k_n}:=\frac{(m!)}{(k_1!)\cdots (k_n!)}.
\]
Assume that $n\ge 1$. Set $k_{n+1}=p-1-k_1-\cdots-k_n$. By Lemma \ref{L:SP}, we have
\begin{align*}
\sum_{0\le i\le p-1} (g^ix^{p-1})^{[n+1]}&\\
=\sum_{\substack{0\le i\le p-1\\ 0\le k_1,\ldots,k_n\le p-1}} &\Bigg(g^ic_{k_1+\cdots+k_n,k_{n+1}}g^ic_{k_1+\cdots+k_{n-1},k_n} \cdots g^ic_{k_1,k_2}g^ic_{0,k_1}\Bigg)\\
=\sum_{\substack{0\le i\le p-1\\ 0\le k_1,\ldots,k_n\le p-1}} &\Bigg({p-1\choose k_1+\cdots+k_n}{k_1+\cdots+ k_n\choose k_n}\cdots {k_1+k_2\choose k_1}\\
g^i(g^{k_1+\cdots+k_n}&x^{k_{n+1}})g^i(g^{k_1+\cdots+k_{n-1}}x^{k_n})\cdots g^i(g^{k_1}x^{k_2})g^i(x^{k_1})\Bigg)\\
=\sum_{\substack{0\le i\le p-1\\ 0\le k_1,\ldots,k_n\le p-1}}&{p-1\choose k_1,\ldots,k_{n+1}}
g^{(n+1)i+nk_1+(n-1)k_2+\cdots+k_n}x^{p-1}.
\end{align*}
Therefore,
\begin{align*}
\nu_{n+1}\left(R(p)\right)&=\\
\sum_{0\le k_1,\ldots,k_{n+1}\le p-1} &{p-1\choose k_1,\ldots,k_{n+1}}
 \delta_{gx^{p-1}}
\left(\sum_{i=0}^{p-1}g^{(n+1)i+nk_1+(n-1)k_2+\cdots+k_n}x^{p-1}\right).
\end{align*}
Suppose the indices $k_1,k_2,\ldots, k_n$ are fixed. Then the inner summation of the above equation becomes
\begin{align*}
\sum_{0\le i\le p-1}&g^{(n+1)i+nk_1+(n-1)k_2+\cdots+k_n}x^{p-1}\\
&=\displaystyle{\begin{cases}
p\left(g^{(nk_1+(n-1)k_2+\cdots+k_n)}x^{p-1}\right)=0 & \text{if}\ p\mid n+1\\
\left(1+g+\cdots+g^{p-1}\right)x^{p-1}& \text{if}\ p\nmid n+1.
\end{cases}}
\end{align*}
In a conclusion, by Fermat's Little Theorem and for $n\geq 1$, we have
\begin{align*}
\nu_{n+1}\left(R(p)\right)=
\begin{cases}
0 & \text{if}\ p\mid n+1\\
\sum_{k_1,\cdots,k_{n+1}} {p-1\choose k_1,\cdots,k_{n+1}}=(n+1)^{p-1}=1  & \text{if}\ p\nmid n+1 .
\end{cases}
\end{align*}
Note that $\nu_{1}\left(R(p)\right)=1$. Therefore, we showed that
\[\Big\{\nu_n(R(p))\Big\}_{ n\geq 1}=\Big\{\underbrace{1, \ldots, 1}_{p-1}, 0, \underbrace{1, \ldots, 1}_{p-1}, 0, \ldots\Big\}.\]

\end{proof}

\begin{prop}
The minimal polynomial of the sequence $\big\{\nu_n(R(p))\big\}$ is \[f(x)=x^p-1.\]
\end{prop}

\begin{proof}
The first $p+1$ terms of $\big\{\nu_n(R(p))\big\}$ are $1,\, \ldots,\, 1,\, 0,\, 1$. The degree of the minimal polynomial can not be less than $p$. Otherwise, $\big\{\nu_n(R(p))\big\}$ satisfies a polynomial $f(x)=f_0+f_1x_1+\cdots+f_{p-1}x^{p-1}$. Then
\[A[f_0\, f_1\, \ldots\, f_{p-1}]^{\mathrm{T}}=0,\]
where $A$ is the matrix with 0's on the anti-diagonal and 1's elsewhere. Note that the determinant of $A$ is $p-1$ or $-(p-1)$. This implies that $f_0= f_1= \cdots= f_{p-1}=0$, contradiction. Hence the degree of the minimal polynomial is at least $p$. One can verify that $\big\{\nu_n(R(p))\big\}$ satisfies the polynomial $f(x)=x^p-1$. This completes the proof.
\end{proof}

\end{document}